\documentclass[reqno,11pt]{amsart}
\usepackage[margin=1in]{geometry}

\usepackage{amsthm, amsmath, amssymb, bm}
\usepackage{tikz}

\usepackage{microtype} 

\usepackage[utf8]{inputenc}
\usepackage[T1]{fontenc}

\usepackage[textsize=scriptsize,backgroundcolor=orange!5]{todonotes}

\usepackage[hidelinks]{hyperref}
\usepackage{url}

\usepackage[noabbrev,capitalize]{cleveref}
\crefname{equation}{}{}

\usepackage{xcolor,graphics}

\numberwithin{equation}{section}

\newtheorem{theorem}{Theorem}[section]

\newtheorem{lemma}[theorem]{Lemma}

\newtheorem*{question*}{Question} \Crefname{question}{Question}{Questions}

\theoremstyle{definition}

\theoremstyle{remark}
\newtheorem*{remark}{Remark}

\newcommand{\norm}[1]{\left\lVert#1\right\rVert}

\newcommand{\paren}[1]{\left( #1 \right)}

\newcommand{\KL}{\textrm{KL}}

\newcommand{\bx}{\bm{x}}

\newcommand{\RR}{\mathbb{R}}

\newcommand{\NN}{\mathbb{N}}

\newcommand{\bB}{\bm{B}}

\newcommand{\bS}{\bm{S}}

\newcommand{\vol}{\mathop{\text{vol}}}

\title[Superball packing upper bounds]{Exponential improvements for \\ superball packing upper bounds}

\author[Sah]{Ashwin Sah}
\address{Massachusetts Institute of Technology, Cambridge, MA 02139, USA}
\email{asah@mit.edu}

\author[Sawhney]{Mehtaab Sawhney}
\address{Massachusetts Institute of Technology, Cambridge, MA 02139, USA}
\email{msawhney@mit.edu}

\author[Stoner]{David Stoner}
\address{Department of Mathematics, Stanford University, Stanford, CA 94305, USA}
\email{dwstoner@stanford.edu}

\author[Zhao]{Yufei Zhao}
\address{Department of Mathematics, Massachusetts Institute of Technology, Cambridge, MA 02139, USA}
\email{yufeiz@mit.edu}
\thanks{YZ was supported by NSF Awards DMS-1362326 and DMS-1764176, and the MIT Solomon Buchsbaum Fund.}

\begin{document}

\begin{abstract}
We prove that for all fixed $p > 2$, the translative packing density of unit $\ell_p$-balls in $\RR^n$ is at most $2^{(\gamma_p + o(1))n}$ with $\gamma_p < - 1/p$. This is the first exponential improvement in high dimensions since van der Corput and Schaake (1936).
\end{abstract}

\maketitle

\section{Introduction} \label{sec:intro}

The sphere packing problem asks for the densest packing of non-overlapping unit balls in $\RR^n$. This is an old and difficult problem whose exact solution is only known in dimensions 1, 2, 3, 8, and 24. The problem is already non-trivial in two dimensions (see \cite{Hales00} for a short proof). The three-dimensional sphere packing problem is known as Kepler's conjecture, and it was solved by Hales~\cite{Hales05} via a monumental computer-assisted proof. The problem in eight dimensions was recently resolved by Viazovska~\cite{Via17} in a stunning breakthrough, and the method was then quickly extended to solve the problem in twenty-four dimensions~\cite{CKMRV17}. Dimensions 8 and 24 are special due to the existence of highly dense and symmetric lattices known as the $E_8$ lattice (dimension 8) and the Leech lattice (dimension 24). See the survey~\cite{Cohn17NAMS} and its references for background and recent developments.

In this paper, we study translative packings of $\ell_p$-balls in high dimensions. Denote the $\ell_p$-balls with radius $R$ in $\RR^n$ by $\bB_p^n(R) := \{ \bm x \in \RR^n : \norm{\bm x}_p \le R\}$ and the unit $\ell_p$-ball by $\bB_p^n:=\bB_p^n(1)$. Here $\norm{(x_1, \dots, x_n)}_p := (|x_1|^p + \cdots + |x_n|^p)^{1/p}$ is the $\ell_p$-norm. 
The name \emph{superball} refers to $\ell_p$-balls with $p > 2$ \cite{RS87}.  Superballs are more cube-like compared to the familiar $\ell_2$-balls. See \cite{JST09,JST10,DV18} for studies of $\ell_p$-ball packings in $\RR^3$. Although $\ell_p$-balls do not possess rotational symmetry, in this paper we only consider translations of identical $\ell_p$-balls, not allowing rotations. The best known lower bounds on high dimensional superball packing densities do not use rotations~\cite{EOR91} (see \cref{sec:review}).

Let $\Delta_p(n)$ denote the maximum translative packing density of copies of $\bB_p^n$ in $\RR^n$. Here \emph{density} is the fraction of space occupied by these balls. For fixed $p \in [1,\infty)$, let
\[
\gamma_p := \limsup_{n\to\infty} \frac1n \log_2 \Delta_p(n)
\]
be the exponential rate of optimal packing densities in high dimensions. The precise value of $\gamma_p$ is unknown for any $p \in [1,\infty)$, and the current best  upper and lower bounds are quite far apart. For Euclidean balls, $p =2$, the best high dimensional upper bound (apart from constant factors) is due to Kabatiansky and Levenshtein~\cite{KL78}:
\[
\Delta_2(n) \le 2^{(\kappa_{\KL} + o(1)) n}, \quad \text{where } \kappa_{\KL} := -0.5990\dots.
\]
See Cohn and Zhao~\cite{CZ14} and Sardari and Zargar~\cite{SZ20}
for constant factor improvements over Kabatiansky and Levenshtein~\cite{KL78}. 
For lower bounds, we have $\Delta_p(n) \ge 2^{-n}$ for all $n$ and $p \ge 1$ since every maximal packing has density at least $2^{-n}$. For $p=2$, there have only been subexponential improvements, with the current best lower bound due to Venkatesh~\cite{Ven13}.
In summary, the best bounds on $\gamma_2$ are $-1 \le \gamma_2 \le \kappa_{\KL} = -0.5990\dots$.


\medskip

For $p > 2$, the current best upper bound on the exponential rate of superball packing densities was first proved by van der Corput and Schaake~\cite{CS36} via Blichfeldt's method~\cite{Bli29} (e.g., see \cite[Section 6.3]{Zong}), giving
\[
\gamma_p \le -1/p \quad \text{for } p > 2.
\]
There have been subsequent subexponential upper bound improvements on $\Delta_p(n)$ for $p > 2$, e.g., Rankin \cite{RankinI, RankinII}. We defer to \cref{sec:remarks} for a discussion of known bounds on $\gamma_p$ in other regimes.

In this paper, we prove a new upper bound on $\gamma_p$ for all $p > 2$, giving the first exponential improvement since 1936 on the upper bound of superball packing densities in high dimensions.

\begin{theorem} \label{thm:main}
For all $p \ge 2$,
\[
\gamma_p \le \inf_{0 < \theta < \pi/2}\left(\frac{1 + \sin\theta}{2\sin\theta}\log_2\frac{1 + \sin\theta}{2\sin\theta} - \frac{1 - \sin\theta}{2\sin\theta}\log_2\frac{1 - \sin\theta}{2\sin\theta}+\frac{2}{p}\log_2\sin\frac{\theta}{2}\right).
\]
In particular, $\gamma_p < -1/p$ for all $p \ge 2$.
\end{theorem}

See \cref{fig:bounds} for a plot of the bounds.

\begin{figure}
    \centering
    \includegraphics{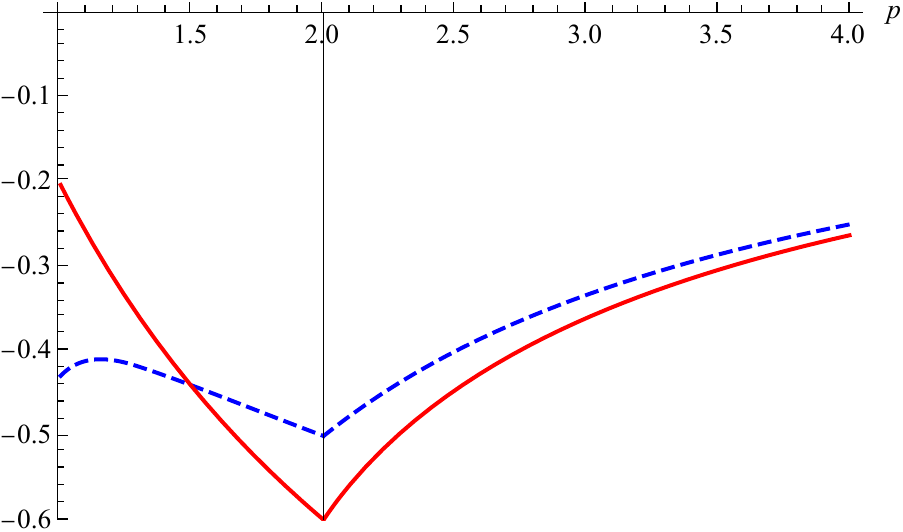}
    \caption{Upper bounds on the exponential rate $\gamma_p$ of translative packing densities of identical $\ell_p$-balls in high dimensions. 
    For $p> 2$, the dashed blue curve is the previous upper bound $-1/p$ and the solid red curve is our new upper bound.
    For $1 \le p < 2$, discussed in \cref{sec:remarks}, the dashed blue curve is \eqref{eq:RankinIII} due to Rankin~\cite{RankinIII} and the solid red curve is \eqref{eq:p<2-KL} derived from the Kabatiansky--Levenshtein~\cite{KL78} sphere packing bound.}
    \label{fig:bounds}
\end{figure}

\begin{remark}
	\cref{thm:main} with $p = 2$ recovers $\gamma_p \le \kappa_{\KL}$. Our upper bound on $\gamma_p$ is continuous with $p$, whereas the previous best bounds were not continuous\footnote{It is unknown whether $p \mapsto \gamma_p$ is continuous. \cref{lem:gamma-monotonicity} implies that $\gamma_p$ is continuous at all but at most countably many points.} at $p=2$. The fact that our bound at $p=2$ recovers the Kabatiansky--Levenshtein bound is not a coincidence, as our proof relies on the Kabatiansky--Levenshtein bound for spherical codes.
\end{remark}

\section{Proof of main theorem}\label{sec:proof}

\subsection{Kabatiansky--Levenshtein spherical code bound}
Denote the $\ell_p$-sphere in $\RR^n$ of radius $R$ by $\bS^{n - 1}_p(R) := \{\bx \in \RR^n : \norm{\bx}_p = R\}$ and the unit $\ell_p$-sphere by $\bS_p^{n - 1} := \bS_p^{n - 1}(1)$. Let $A_p(n, d)$ to be the maximum number of points on $\bS_p^{n - 1}$ with pairwise $\ell_p$-distance at least $2d$, i.e, an $\ell_p$-spherical code. Note that $A_p(n,d) = 1$ unless $d \in [0,1]$. Note that $A_2(n, \sin(\theta/2))$ is the maximum size of a spherical code in $\RR^n$ with pairwise angle at least $\theta$. Kabatiansky and Levenshtein~\cite{KL78} proved that for all\footnote{A simple geometric argument (see \cite[(17)]{Lev75}) shows that the upper bound \eqref{eq:kl-bound} can be improved for $\theta < \theta_\KL :=  1.0995\dots$ to $a(\theta_\KL) + \log_2\sin(\theta_\KL/2) - \log_2\sin(\theta/2)$, but this improvement does not benefit our bounds.} $0 < \theta <\pi/2$,
\begin{equation} \label{eq:kl-bound}
	\limsup_{n\to\infty}\frac{1}{n}\log_2 A_2(n,  \sin(\theta/2) )
	\le
	a(\theta)
\end{equation}
where
\[
a(\theta) := \frac{1 + \sin\theta}{2\sin\theta}\log_2\frac{1 + \sin\theta}{2\sin\theta} - \frac{1 - \sin\theta}{2\sin\theta}\log_2\frac{1 - \sin\theta}{2\sin\theta}.
\]
A projection argument (see \cite[Section 2]{CZ14}) shows that 
\[
\Delta_2(n) \le \sin^n(\theta/2) A_2(n+1, \sin(\theta/2)),
\]
so \eqref{eq:kl-bound} gives
\[
\lim_{n \to \infty} \frac{1}{n}\log_2\Delta_2(n) \le a(\theta) + \log_2\sin\frac\theta2.
\]
The bound $\gamma_2 \le \kappa_\KL = -0.5990\dots$ is obtained by choosing $\theta = \theta_\KL = 1.0995\dots$ to minimize the upper bound above.

\subsection{$\ell_p$-twist}
Fix $p \ge 2$. Define
\[
x^{*} := \operatorname{sgn}(x)|x|^{p/2}, \qquad x \in \RR.
\]
For $\bm x = (x_1, \dots, x_n) \in \RR^n$, write $\bm x^* := (x_1^*, \dots, x_n^*)$, and for $X \subseteq \RR^n$, write $X^* := \{\bm x^* : \bm x \in X\}$.

Observe that for all $x,y \in \RR$,
\begin{equation} \label{eq:transfer}
|x^* - y^*|\ge 2^{1 - p/2}|x - y|^{p/2}.
\end{equation}
Indeed, without loss of generality it suffices to consider two cases: $x \ge 0 \ge y$ and $x \ge y \ge 0$. The former case is an immediate consequence of H\"older's inequality (or the convexity of $x \mapsto x^{p/2}$). In the latter case, we have
\[
x^{p/2} - y^{p/2} \ge (x-y)^{p/2} \ge 2^{1-p/2} (x-y)^{p/2}.
\]
Here we use $(w+z)^{p/2} \ge w^{p/2} + z^{p/2}$ for $w,z > 0$, which can be proved by first normalizing to $w+z=1$ and noting that $w^{p/2} + z^{p/2} \le w + z = 1$.

\begin{lemma} \label{lem:Ap-A2}
	For all $p \ge 2$ and $d \in (0,1]$, we have $A_p(n, d)\le A_2(n, d^{p/2})$.
\end{lemma}
\begin{proof}
Let $X \subseteq \bS_p^{n-1}$ with $|X| = A_p(n,d)$ and $\norm{\bm x - \bm y}_p \ge 2d$ for all distinct $\bm x, \bm y \in X$. We have $\norm{\bm x^*}_2 = \norm{\bm x}_p = 1$ for all $\bm x \in X$, so $X^* \subseteq \bS_2^{n-1}$. For distinct $\bm x, \bm y\in X$, we have
\[
\norm{\bm x^* - \bm y^*}_2^2 = \sum_{i = 1}^n |x_i^* - y_i^*|^2
\ge 2^{2 - p}\sum_{i = 1}^n |x_i - y_i|^p
\ge 2^2d^p,
\]
by \eqref{eq:transfer}. Thus $X^*$ is a subset of $\bS_2^{n-1}$ whose points have pairwise $\ell_2$-distance at least $2d^{p/2}$. Hence $|X| = |X^*| \le A_2(n, d^{p/2})$.	
\end{proof}

\begin{remark}
The same argument shows that $A_p(n, d)\le A_q(n, d^{p/q})$ for all $1\le q\le p$ and $d\in (0, 1]$.
\end{remark}


\begin{lemma}\label{lem:vertical-lp-project}
For every $p \ge 1$, $d \in (0,1]$, and $n \in \NN$, we have $\Delta_p(n)\le d^n A_p\left(n+1, d\right)$.
\end{lemma}
\begin{proof}
Let $\rho < \Delta_p(n)$ be arbitrary.
Consider a translative packing $\{\bm x + \bB_p^n(d): \bm x \in X\}$ in $\RR^n$ with density greater than $\rho$, where $X \subseteq \RR^n$ is the set of centers of the $\ell_p$-balls. By an averaging argument\footnote{A uniform random translation of a unit $\ell_p$-ball inside $[-R,R]^n$ contains more than $d^{-n} \rho + o_{R \to \infty}(1)$ points of $X$.}, there exists some translate of a unit $\ell_p$-ball that contains at least $d^{-n} \rho$ points of $X$. Translating $X$ if necessary, we may assume that $|X \cap \bB_p^n| \ge d^{-n} \rho$. Add an $(n+1)$-st coordinate to each point in $X \cap \bB_p^n$ to obtain a set $X'$ of points on the unit $\ell_p$-sphere in $\RR^{n+1}$. In other words, $X'$ is obtained by projecting the points of $X$ contained in the unit ball ``upward'' to the hemisphere one dimension higher. Since the points in $X$ are pairwise at least $2d$ apart in $\ell_p$-distance, the same holds for $X'$. So $X'$ is an $\ell_p$-spherical code whose points are pairwise separated by $\ell_p$-distance at least $2d$, and hence $d^{-n} \rho \le |X \cap \bB_p^n| = |X'| \le A_p(n+1,d)$. Since $\rho$ can be arbitrarily close to $\Delta_p(n)$, we obtain the claimed inequality.
\end{proof}

\begin{remark}
As in \cite{CZ14}, the above argument can be modified so that we do not need to add a new dimension	when $d \in [1/2,1]$, resulting in a slightly better bound $\Delta_p(n) \le d^{-n} A_p(n,d)$. We omit the details of this modification since this improvement does not affect the exponential asymptotics.
\end{remark}

\begin{proof}[Proof of \cref{thm:main}]
	Applying \cref{lem:vertical-lp-project,lem:Ap-A2}, we have, for every $0 < \theta < \pi/2$,
	\[
	\Delta_p(n) \le \sin(\theta/2)^{2n/p} A_p(n+1, \sin(\theta/2)^{2/p}) \le \sin(\theta/2)^{2n/p}A_2(n+1, \sin(\theta/2)).
	\]
	Applying \eqref{eq:kl-bound}, we obtain
	\[
	\gamma_p = \limsup_{n \to \infty} \frac1n \log_2 \Delta_p(n) \le 
	a(\theta)  + \frac{2}{p}\log_2\sin(\theta/2).
	\]
	The main result follows by taking the infimum of the bound over $\theta \in (0,\pi/2)$.
	
	Setting $\theta = \pi/2 - \eta$, we have, with $p \ge 2$ fixed and $\eta \to 0^+$,
	\[
	\gamma_p \le -\frac{1}{p} - \frac{\eta}{p\ln 2} + o(\eta).
	\]
	So choosing $\eta > 0$ sufficiently small gives $\gamma_p < -1/p$ for all $p\ge 2$.
\end{proof}

\section{Remarks} \label{sec:remarks}

\subsection{Asymptotics}
Setting $\theta = \pi/2 - (p \ln p)^{-1}$, we obtain
\[
\gamma_p\le -\frac{1}{p} - \frac{1}{\ln 4}\cdot\frac{1}{p^2\ln p} + O\left(\frac{1}{p^2\ln^2 p}\right), \quad \text{as } p\to \infty.
\]
Taking $\theta = \theta_\KL$ gives
\[
\gamma_p\le \kappa_\KL + \frac{2 - p}{p}\log_2\sin\frac{\theta_{KL}}{2}, \quad \text{for all } p \ge 2.
\]
Thus, as $\epsilon \to 0^+$,
\[
\gamma_{2 + \epsilon}\le\gamma_\KL - \left(\frac{1}{2}\log_2\sin \frac{\theta_\KL}{2}\right)\epsilon + O(\epsilon^2) = (-0.5990\dots) + (0.4650\dots)\epsilon + O(\epsilon^2).\]

\subsection{Review of other bounds on $\gamma_p$} \label{sec:review}
Here we survey other existing bounds on $\gamma_p$.

\medskip

For $p = 2$, the best known bounds are $-1 \le \gamma_2 \le \kappa_\KL = -0.5990\dots$ as discussed earlier.

\medskip

For $p > 2$, the best known upper bounds are the ones given in this paper. For lower bounds, extending on methods developed by Rush~\cite{Rush89} and Rush--Sloane~\cite{RS87}, Elkies, Odlyzko, and Rush~\cite{EOR91} proved $\gamma_p > -1$ for all $p > 2$, thereby exponentially beating the Minkowski--Hlawka lower bound. See \cite{EOR91} for the precise bound. Their bounds have the following asymptotics:
\[
\gamma_p \ge - (1+o(1))\frac{\ln \ln p}{p \ln 2}, \quad \text{as } p \to \infty,
\]
and
\[
\gamma_{2+\epsilon} \ge -1 +  \paren{\frac{\sqrt{\pi} \zeta(3)}{2 \ln 2} + o(1)}  \frac{\epsilon}{\ln^{3/2}(1/\epsilon)}, \quad \text{as } \epsilon \to 0^+.
\]
Here $\zeta$ denotes the Riemann zeta function. See \cite{LX08} for some later improvements using algebraic-geometric codes for some specific integers $p$.

\medskip

For $1\le p < 2$, no improvement over the Minkowski--Hlawka lower bound $\gamma_p \ge -1$ is known. The best upper bound on $\gamma_p$ is due to Rankin~\cite{RankinIII}, based on Blichfeldt's method~\cite{Bli29}:
\begin{equation}
    \label{eq:RankinIII}
    \gamma_p \le 
    \inf_{\frac12 \leq \frac1q \le \frac13 \paren{ 1 + \frac1p}} \left( b(p) - b(q) - 1 + 1/p + (1/q - 1/p) \log_2\left( \frac{2-1/q}{1-1/q}\right) \right)
\end{equation}
where
\begin{equation}
	\label{eq:b}
b(p) := \lim_{n\to\infty} \frac{1}{n} \log_2 \vol \bB_p^n(n^{1/p}) = 1 + 
  \log_2 \Gamma\left(1 + \frac{1}{p}\right) + \frac{1}{p} \log_2 (pe).
\end{equation}
Recall that $\vol \bB_p^n= 2^n\Gamma(1+1/p)^n/\Gamma(1+n/p)$. 

For packings of congruent cross-polytopes (i.e., unit $\ell_1$-balls) allowing rotations,
Fejes T\'oth, Fodor, and V\'igh \cite{TFV15} proved an exponentially decaying upper bound in high dimensions. For translative packing of unit $\ell_1$-balls, the upper bound \cref{eq:RankinIII} remains best known in high dimensions.

We note that the above bound \eqref{eq:RankinIII} can be improved on the region $p \in [1.494\dots, 2)$ using the Kabatiansky--Levenshtein bound via the following folklore observation.

\begin{lemma} \label{lem:gamma-monotonicity}
For $1 \le p \le q \le \infty$, $\gamma_p - b(p) \le \gamma_q - b(q)$.
\end{lemma}
\begin{proof}
By monotonicity of norms, we have $n^{-1/p}\|\bm x\|_p \leq n^{-1/q} \|\bm x\|_q$, so 
$\bB_p^n(n^{1/p}) \supseteq \bB_q^n(n^{1/q})$. Any packing of $\bB_p^n(n^{1/p})$ can be shrunk into a packing of $\bB_q^n(n^{1/q})$. Hence
\[
\frac{\Delta_{p}(n)}{\vol \bB_p^n(n^{1/p})}
\le \frac{\Delta_{q}(n)}{\vol \bB_p^n(n^{1/q})}.
\]
Taking log, dividing by $n$, and letting $n \to \infty$ yields the lemma.
\end{proof}

Using $\gamma_2 \le \kappa_{\textrm{KL}}$, we find that
\begin{equation} \label{eq:p<2-KL}
	\gamma_p \le \kappa_{\textrm{KL}} - b(2) + b(p) = (-0.5990\dots) - b(2) + b(p) \quad \text{for } 1 \le p < 2.
\end{equation}
Thus
\[
\gamma_p \le \min\{ 
\text{RHS of \eqref{eq:RankinIII}},
\text{RHS of \eqref{eq:p<2-KL}}
\}
\]
See \cref{fig:bounds} for an illustration of the above bounds.

\section*{Acknowledgments}

This work began during Y.Z.'s internship at Microsoft Research New England, and he would like to thank Henry Cohn for discussions and mentorship and Microsoft Research for its hospitality.


\end{document}